\documentclass[a4paper,12pt]{article}

\usepackage{graphicx}        
\usepackage{multicol}        
\usepackage[bottom]{footmisc}
\usepackage{url}

\usepackage{amsmath}
\usepackage{amsthm}
\usepackage{amssymb}
\usepackage{manfnt}

\newcommand{\pt}{\partial}

\newcommand{\laplace}{\Delta}

\newcommand{\eps}{\varepsilon}
\newcommand{\R}{\mathbb{R}}
\newcommand{\E}{\mathrm{e}}
\newcommand{\wom}{{\widehat \Omega}}
\newcommand{\C}{c\,}
\newcommand{\LL}{\mathcal{L}}

\newcommand{\norm}[2]{\|{#1}\|_{#2}}

\newcommand{\ord}[1]{\mathcal{O}\left(#1\right)}
\renewcommand{\tilde}{\widetilde}

\newcommand{\ve}[1]{\boldsymbol{#1}}

\newcommand{\cube}{{\mbox{\scriptsize\mancube}}}

\theoremstyle{plain}
\newtheorem{theorem}{Theorem}
\newtheorem{lemma}[theorem]{Lemma}

\theoremstyle{remark}
\newtheorem{remark}[theorem]{Remark}

\begin{document}

\title{On the sharpness of Green's function estimates for a convection-diffusion
       problem%
       \footnote{This work has been supported by Science Foundation Ireland
                 under the Research Frontiers Programme 2008;
                 Grant 08/RFP/MTH1536.
                 The second author was also supported
                 by the Bulgarian Foundation for Science; project DID 02/37-2009.}}

\author{Sebastian~Franz\footnote{
         Institut f\"ur Numerische Mathematik,
         Technische Universit\"at Dresden,
         01062 Dresden, Germany\newline
         e-mail: sebastian.franz@tu-dresden.de}
        \and
        Natalia~Kopteva\footnote{
         Department of Mathematics and Statistics,
         University of Limerick,
         Limerick,
         Ireland\newline
         e-mail: natalia.kopteva@ul.ie}
        }


   \maketitle
   \abstract{
   Linear singularly perturbed convection-diffusion
    problems with characteristic layers are considered in three dimensions.
We demonstrate the sharpness of our recently obtained upper bounds for the associated Green's
    function and its derivatives in the $L_1$ norm.
For this, in this paper we establish
the corresponding lower bounds.
Both upper and lower bounds explicitly show any dependence
on the singular perturbation parameter.}


    \vspace{-0.1cm}

   \section{Introduction}
%
Consider the convection-diffusion problem in the domain $\Omega=(0,1)^3$:\vspace{-0.1cm}
   \begin{subequations}\label{eq:Lu}
   \begin{align}
     \LL_{\ve{x}}u(\ve{x})=-\eps\laplace_{\ve{x}}u(\ve{x})-2\alpha\,\pt_{x_1} u(\ve{x})
     &=f(\ve{x})&&\mbox{for }\ve{x}\in\Omega,\\
     u(\ve{x})&=0&&\mbox{for }\ve{x}\in\partial\Omega.
   \end{align}
   \end{subequations}
   Here $\eps\in(0,1]$ is a small positive parameter, while
   $\alpha$ is a positive constant. Then \eqref{eq:Lu}
   is a singularly perturbed convection-dominated problem,
   whose solutions typically exhibit sharp characteristic boundary and interior layers.

   This article addresses the sharpness of our recently published obtained
   upper bounds for the associated Green's
    function and its derivatives in the $L_1$ norm.
   Our interest in considering the Green's function of problem
   is motivated by the numerical analysis of this computationally
   challenging problem.
   More specifically, these estimates will be used in the forthcoming
   paper \cite{FK10_NA} to derive robust a posteriori
   error bounds for computed solutions of this problem using
   finite-difference methods.
   (This approach is related to recent articles \cite{Kopt08,CK09},
   which address the numerical solution of singularly perturbed
   equations of reaction-diffusion type.)
   In a more general numerical-analysis context, 
   we note that sharp estimates for continuous Green's functions (or
   their generalised versions) frequently play a crucial role in a
   priori and a posteriori error analyses \cite{erikss,Leyk,notch}.

 For each fixed $\ve{x}\in\Omega$, the Green's function $G$
 associated with (\ref{eq:Lu}) satisfies
   \begin{subequations}\label{eq:Green_adj}
   \begin{align}
     \hspace{-0.3cm}
     \LL^*_{\ve\xi}G(\ve{x};\ve\xi)
          :=-\eps\laplace_{\ve\xi}G(\ve{x};\ve\xi)+2\alpha\, \pt_{\xi_1}\!G(\ve{x};\ve\xi)
         &=\delta(\ve{x}-\ve\xi)&&
    \mbox{for }\ve\xi\in\Omega,\\
     G(\ve{x};\ve\xi)&=0&&\mbox{for }\ve\xi\in\partial\Omega.
   \end{align}
   \end{subequations}
   Here $\LL^*_{\ve\xi}$ is the adjoint differential operator to $\LL_{\ve{x}}$,
   and $\delta(\cdot)$ is the three-dimensional Dirac $\delta$-distribution.

Note that the Green's function for a singularly perturbed self-adjoint
reaction-diffusion operator $-\eps\triangle_{\ve{x}}+\alpha$ is almost
radially symmetric and exponentially decaying away from the singular point
 \cite{CK09}.
By contrast, the Green's function for our convection-diffusion problem~(\ref{eq:Lu})
exhibits a strong anisotropic structure, which is demonstrated by Figure~\ref{fig:green}.

   \begin{figure}[tb]
      \centerline{
      \includegraphics[width=0.5\textwidth]{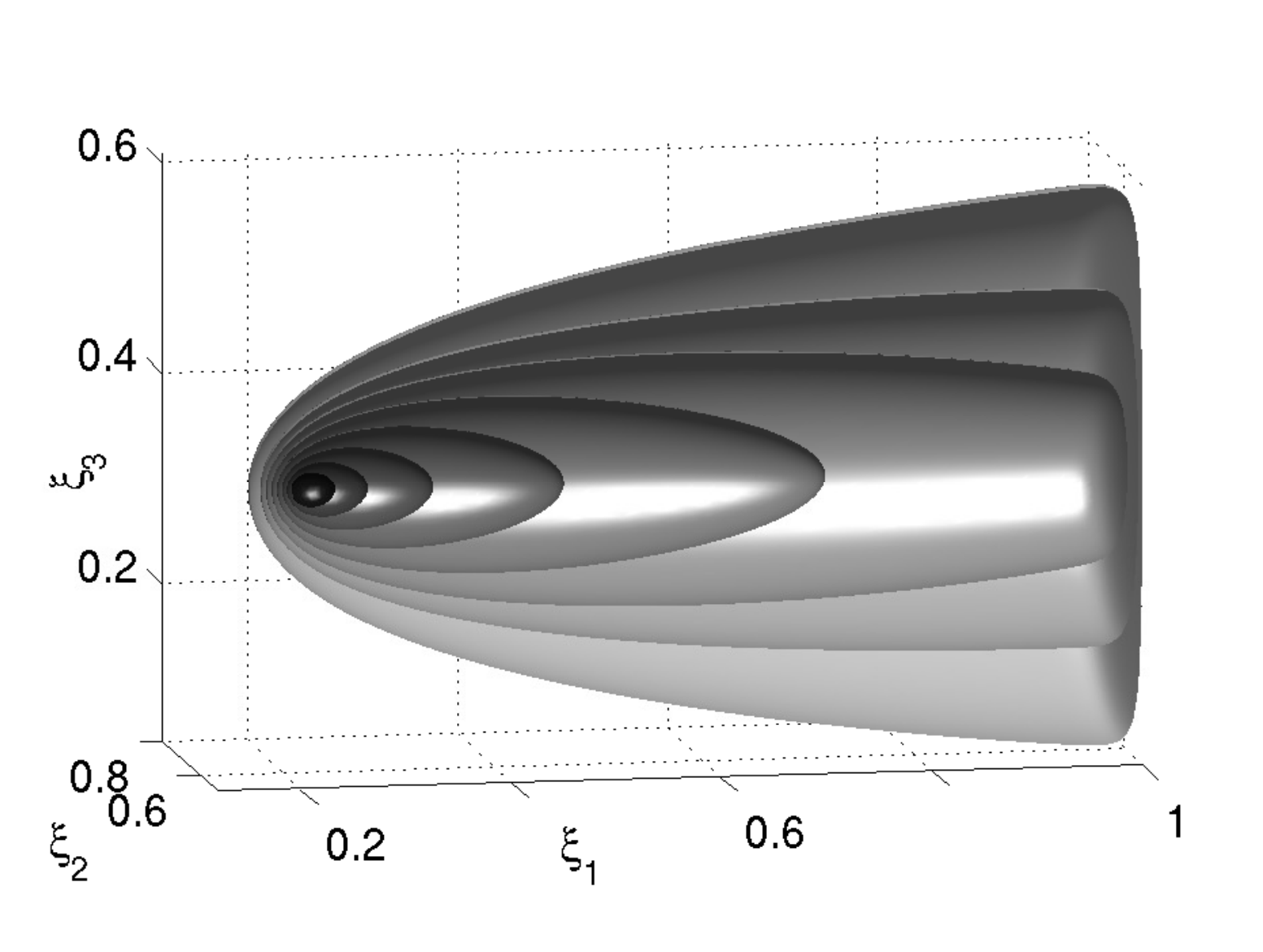}
      \includegraphics[width=0.5\textwidth]{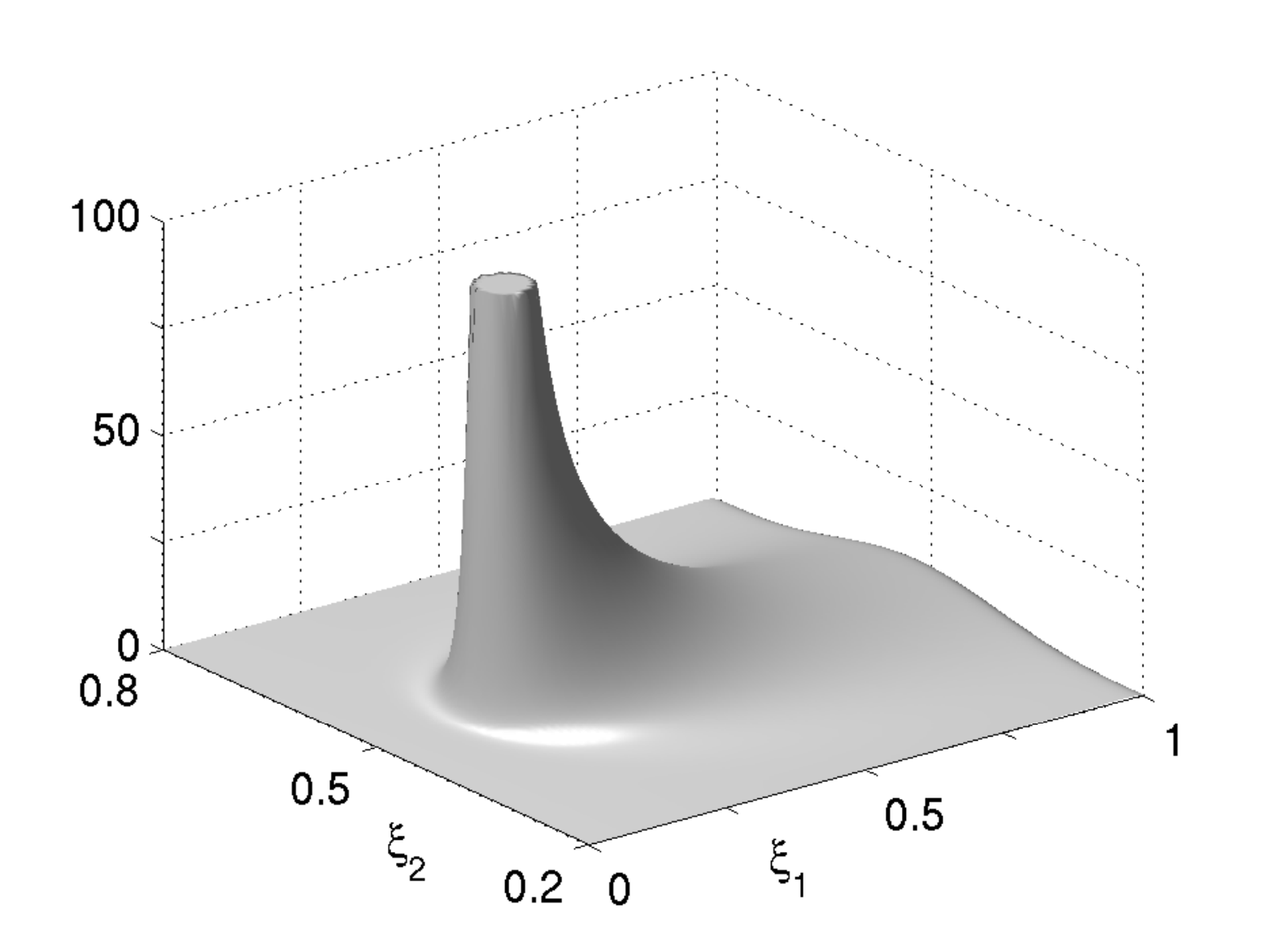}}
      \caption{Anisotropy of the Green's function $G(\ve{x};\ve\xi)$ associated with
      (\ref{eq:Lu}) for
$\eps=0.01$ and $\ve{x}=(\frac15,\frac12,\frac13)$.
Left: 
isosurfaces at values of
   $1,\,4,\,8,\,16,\,32,\,64,\,128$,  and $256$.
   Right: a two-dimensional graph for fixed $\xi_3=x_3$.}
             \label{fig:green}
   \end{figure}

   In \cite{FK10_TR1, FK11_1}
   we have obtained certain upper bounds for the Green's function 
   associated with a variable-coefficient version of
   (\ref{eq:Lu}), which
   we now cite.

   \begin{theorem}[\cite{FK10_TR1, FK11_1}]\label{thm:main_upper}
      Let $\eps\in(0,1]$. The Green's function $G$ associated with~\eqref{eq:Lu}
      on the unit cube $\Omega=(0,1)^3$ satisfies, for all $\ve{x}\in\Omega$,
      the following \underline{upper bounds}:
      \begin{subequations}\label{eq_theorem}
      \begin{align}
         \norm{\pt_{\xi_1} G(\ve{x};\cdot)}{1;\Omega}
            &\leq C(1+|\ln \eps|),\\
         \norm{\pt_{\xi_k} G(\ve{x};\cdot)}{1;\Omega}
            &\leq C\eps^{-1/2},\quad k=2,3.
      \end{align}
      Furthermore,  for any ball $B(\ve{x}',\rho)$
    of radius $\rho$ centred at any  $\ve{x}'\in\Omega$, we have
      \begin{align}
         \norm{G(\ve{x};\cdot)}{1,1;\Omega\cap B(\ve{x}',\rho)}
            &\leq C
           \rho/\eps,
      \end{align}
      while for any ball $B(\ve{x},\rho)$
      of radius $\rho$ centred at $\ve{x}\in\Omega$, we have
      \begin{align}
          \norm{\pt^2_{\xi_1} G(\ve{x};\cdot)}{1;\Omega\setminus B(\ve{x},\rho)}
            &\leq C\eps^{-1}\ln(2+\eps/\rho),\\
          \norm{\pt^2_{\xi_k} G(\ve{x};\cdot)}{1;\Omega\setminus B(\ve{x},\rho)}
            &\leq C\eps^{-1}(|\ln\eps|+\ln(2+\eps/\rho)),\quad k=2,3.
      \end{align}
      \end{subequations}
   \end{theorem}

   \begin{remark}
 Theorem~\ref{thm:main_upper} is given in \cite{FK10_TR1, FK11_1}
   for a more general variable-coefficient operator
$-\eps\laplace_{\ve{x}}-a(\ve{x})\,\pt_{x_1} +b(\ve{x})$ with
sufficiently smooth coefficients that satisfy
$a(\ve{x})\ge2\alpha>0$
and
$b(\ve{x})+\pt_{x_1}a(\ve{x})\ge 0$ for all $\ve{x}\in\bar\Omega$.
   \end{remark}

   The purpose of this paper is to show the sharpness of the bounds
   of Theorem~\ref{thm:main_upper}
   up to a constant $\eps$-independent multiplier
   in the following sense.
   \begin{theorem}\label{thm:main_lower}
Let $\eps\in(0,c_0]$ for some sufficiently small positive $c_0$.
The Green's function $G$ associated with~the constant-coefficient
problem \eqref{eq:Lu}
      in the unit cube $\Omega=(0,1)^3$ satisfies,
      for all $\ve{x}\in[\frac14,\frac34]^3$,
      the following \underline{lower bounds}:%
      \begin{subequations}\label{eq_thm:main_lower}
      \begin{align}
         \norm{\pt_{\xi_1} G(\ve{x};\cdot)}{1;\Omega}
            &\geq \C|\ln \eps|,\\
         \norm{\pt_{\xi_k} G(\ve{x};\cdot)}{1;\Omega}
            &\geq \C\eps^{-1/2},\quad k=2,3.
            \intertext{Furthermore, for any ball $B(\ve{x};\rho)$ of radius $\rho\le\frac18$, we have}
         \norm{G(\ve{x};\cdot)}{1,1;\Omega\cap B(\ve{x},\rho)}
            &\geq \begin{cases}
                                 \C\rho/\eps, & \mbox{if~}\rho\le 2\eps,\\
                                          \C(\rho/\eps)^{1/2},&\mbox{otherwise},\\
                                       \end{cases}\\
         \norm{\pt^2_{\xi_1} G(\ve{x};\cdot)}{1;\Omega\setminus B(\ve{x},\rho)}
                       &\geq \C\eps^{-1}\ln(2+\eps/\rho),
           \quad&&\mbox{if~}\rho\le c_1\eps,
            \\
            \norm{\pt^2_{\xi_k} G(\ve{x};\cdot)}{1;\Omega\setminus B(\ve{x},\rho)}
            &\geq \C\eps^{-1}(\ln(2+\eps/\rho)+|\ln\eps|),
           &&\mbox{if~}\rho\le{\textstyle\frac18},
      \end{align}
      where $\;k=2,3$, and $c_1$ is a sufficiently small positive constant.
      \end{subequations}
   \end{theorem}

   \begin{remark}
     The restriction $\ve{x}\in[\frac14,\frac34]^3$ is somewhat arbitrary
     and can be replaced by $\ve{x}\in[\theta,1-\theta]^3$ with
     any $\eps$-independent constant $\theta\in(0,\frac12)$
     (then $\rho\le\frac12\theta$ is imposed instead of $\rho\le\frac18$).
   \end{remark}

 The paper is structured as follows.
   Sharp lower
   bounds for the fundamental solution in $\R^3$
   are derived in Section~\ref{sec:fundamental}.
   The next Section~\ref{sec:approx} is devoted to
the proof of
   Theorem~\ref{thm:main_lower} in a simpler setting, for the domain $(0,1)\times\R^2$.
Sections~\ref{sec_bounded} and ~\ref{sec:reac} briefly describe
 an extension of the analysis of
Section~\ref{sec:fundamental} to the domain $\Omega=(0,1)^3$
and to convection-reaction-diffusion problems.
   Finally, in Section~\ref{sec:outlook}, we give an outlook
   for problems in $n$ dimensions.

\textit{Notation.} Throughout the paper,
   $C$ denotes a generic positive constant, typically sufficiently large,
    while $\C$ denotes a sufficiently small generic positive constant;
    they take different values in different formulas, but are
    independent of the singular perturbation parameter $\eps$.
   The usual Sobolev spaces $W^{m,p}(D)$ and $L_p(D)$
   on any measurable
  set $D\subset\R^3$ are used; 
   the $L_p(D)$ norm is denoted by $\norm{\cdot}{p;D}$,
   while the $W^{m,p}(D)$ norm is denoted
   by $\norm{\cdot}{m,p;D}$.
   By $\ve{x}=(x_1,x_2,x_3)$ we denote an element of $\R^3$.
   For an open ball in $\R^3$, we employ the notation
   $B(\ve{x'},\rho)=\{\ve{x}\in\R^3: \sum_{k=1}^3(x_k-x'_k)^2<\rho^2\}$.
   The notation $\pt_{x_1} f$, $\pt^2_{x_1} f$ and $\laplace_{\ve{x}}$
   is used for the
   first- and second-order partial derivatives
   of a function $f$ in variable $x_1$, and
 the Laplacian in variable $\ve{x}$, respectively.
%
\section{The fundamental solution}\label{sec:fundamental}
%
In this section we investigate the fundamental solution $g$ that
solves a similar problem to \eqref{eq:Green_adj}
but posed in the domain $\R^3$:
   \begin{align}
     \LL^*_{\ve\xi} g(\ve{x};\ve\xi)
         =-\eps\laplace_{\ve\xi} g(\ve{x};\ve\xi)+2\alpha\,\pt_{\xi_1} g(\ve{x};\ve\xi)
        &=\delta(\ve{x}-\ve\xi)\quad\mbox{for}\; \ve\xi\in\R^3.\label{eq:Green_adj_const}
   \end{align}
   Using \cite{CK09,KSh87} (see also \cite{FK10_TR1, FK11_1}),
   a calculation shows that
   the fundamental solution $g$ is explicitly represented by
   \begin{gather}\label{eq:def_g0}
     g=g(\ve{x};\ve\xi)
      =\frac{1}{4\pi\eps^2}\, \frac{\E^{\alpha(\widehat\xi_{1,[x_1]}-\widehat r_{[x_1]})}}
                                   {\widehat r_{[x_1]}}
   \end{gather}
   with the scaled variables
   $\widehat\xi_{1,[x_1]}=\frac{\xi_1-x_1}{\eps}$,
   $\widehat\xi_2=\frac{\xi_2-x_2}{\eps}$, $\widehat\xi_3=\frac{\xi_3-x_3}{\eps}$
   and the scaled distance between $\ve{x}$ and $\ve{\xi}$ denoted by
   $\widehat r_{[x_1]}=\sqrt{\widehat\xi_{1,[x_1]}^2+\widehat\xi_2^2+\widehat\xi_3^2}$.
   We use the subindex $[x_1]$ in $\widehat\xi_{1,[x_1]}$ and $\widehat r_{[x_1]}$
   to highlight their dependence  on $x_1$
   as in many places $x_1$ will take different values;
   but when there is no ambiguity, we shall
   simply write $\widehat\xi_1$ and $\widehat r$.

   Next, we evaluate derivatives of $g$  of order one and two:
   \begin{subequations}
   \begin{align}
      \label{g_xi1} \pt_{\xi_1} g
                     &=\frac{\E^{\alpha(\widehat\xi_1-\widehat r)}}{4\pi\eps^3}
                     \,\widehat r^{-2}
                             \left[\alpha\bigl(\widehat r-\widehat \xi_1\bigr)-\frac{\widehat\xi_1}{\widehat r}\right],\\
      \label{g_xi2} \pt_{\xi_k} g
                     &=-\frac{\E^{\alpha(\widehat\xi_1-\widehat r)}}{4\pi\eps^3}
                             \,\bigl(\alpha\widehat r+1\bigr)\,\frac{\widehat\xi_k}{\widehat r^3},\quad k=2,\,3,\\
      \label{g2_xi1_xi1} \pt^2_{\xi_1} g
                           &=\frac{\E^{\alpha(\widehat\xi_1-\widehat r)}}{4\pi\eps^4}
                           \,\widehat r^{-3}\left[ \alpha^2
                           \bigl(\widehat r-\widehat\xi_1\bigr)^2
                                                -\alpha\bigl(\widehat r-\widehat\xi_1\bigr)
                              \Bigl(1+3\frac{\widehat\xi_1}{\widehat r}\Bigr)
                                                +\frac{3\widehat\xi_1^2-\widehat r^2}
                                                      {\widehat r^2}\right],\\
      \label{g2_xi2_xi2} \pt^2_{\xi_k} g
                           &=\frac{\E^{\alpha\bigl(\widehat\xi_1-\widehat r\bigr)}}{4\pi\eps^4}
                           \,\widehat r^{-3}\left[\alpha^2\widehat\xi_k^2+(\alpha\widehat r+1)
                                                  \frac{3\widehat\xi_k^2-\widehat r^2}
                                                       {\widehat r^2}\right],\quad k=2,\,3.
   \end{align}
   \end{subequations}

 \begin{lemma}\label{lem:fundamental}
Let $\Omega_*:=(0,1)\times\R^2$,
$\eps\in(0,c_0]$ for some sufficiently small constant $c_0$,
       and $0<\alpha\le C$.
      Then the function $g$ of \eqref{eq:def_g0}
      satisfies, for any $\ve{x}\in[2\eps,\frac34]\times\R^2$,
      the following bounds
      \begin{subequations}\label{g_bounds}
         \begin{align}
            \label{g_xi_L1}   \norm{\pt_{\xi_1} g(\ve{x};\cdot)}{1\,;\Omega_*}
                                 &\geq \C|\ln\eps|,\\
            \label{g_eta_L1}  \norm{\pt_{\xi_k} g(\ve{x};\cdot)}{1\,;\Omega_*}
                                 &\geq \C\eps^{-1/2},\quad k=2,3.\\
      \intertext{Furthermore, for any ball $B(\ve{x};\rho)$ of radius $\rho\le\frac18$, we have}
            \label{g_eta_ball} \norm{g(\ve{x};\cdot)}{1,1\,;\Omega_*\cap B(\ve{x};\rho)}
                                 &\geq \begin{cases}
                                 \C\rho/\eps, & \mbox{if~}\rho\le 2\eps,\\
                                          \C(\rho/\eps)^{1/2},&\mbox{otherwise},\\
                                       \end{cases}\\
%
         \norm{\pt^2_{\xi_1} g(\ve{x};\cdot)}{1\,;\Omega_*\setminus B(\ve{x};\rho)}
           &\geq \C\eps^{-1}\ln(2+\eps/\rho),
           \quad&&\mbox{if~}\rho\le c_1\eps,\label{g_xi2_L1}\\
         \norm{\pt^2_{\xi_k} g(\ve{x};\cdot)}{1\,;\Omega_*\setminus B(\ve{x};\rho)}
           &\geq \C\eps^{-1}(\ln(2+\eps/\rho)+|\ln\eps|),
           &&\mbox{if~}\rho\le{\textstyle\frac18},\label{g_eta2_L1}
      \end{align}
      where $\;k=2,3$, and $c_1$ is a sufficiently small positive constant.
      \end{subequations}
   \end{lemma}

\begin{remark}
The statement of the above Lemma~\ref{lem:fundamental} is almost identical with the
statement of Theorem~\ref{thm:main_lower} with $G$ replaced by $g$ and
$\Omega_*$ replaced by $\Omega$.
   \end{remark}

   \begin{proof}
Throughout the proof, $x_1$ is fixed, so we employ the notation
$\widehat\xi_1:=\widehat\xi_{1,[x_1]}$ and
$\ve{\widehat\xi}:=(\widehat\xi_1,\widehat\xi_2,\widehat\xi_3)$.
First, we rewrite all integrals  that appear in (\ref{g_bounds})
 in variable
$\ve{\widehat\xi}\in \wom_*:=(-\frac{x_1}{\eps},\frac{1-x_1}{\eps})\times\R^2$
using $d\ve{\xi}=\eps^{3}d\ve{\widehat\xi}$.
Now it suffices to prove the desired lower bounds
on any sub-domain of $\wom_*$.
In particular, we employ the non-overlapping sub-domains $\wom_1$ and $\wom_2$ of
$\wom_*$:
$$
\wom_1:=B(\ve{0},1)\cap
\Bigl\{\widehat\xi_1^2\ge\widehat\xi_2^2+\widehat\xi_3^2\Bigr\},
\;\;
\wom_2:=\Bigl\{\max\{1,\sqrt{\widehat\xi_2^2+\widehat\xi_3^2}\}
                     \leq\widehat\xi_1\leq{\textstyle\frac{1}{4}}\eps^{-1}\Bigr\}
$$
(but other sub-domains of $B(\ve{0},1)$ similar to $\wom_1$ will be considered as well).
The notation $[v]^+:=\max\{v,0\}$ will be used
      for any function $v$.

(i) The bounds \eqref{g_xi_L1}, \eqref{g_eta_L1} will be obtained using
$\wom_2$.
Note that for $\ve{\widehat\xi}\in \wom_2$ one has
 $\widehat\xi_1\leq\widehat r\leq \sqrt{2}
\, \widehat\xi_1$.
 Introduce the new variables $\psi_k:=\widehat\xi_k/\sqrt{2\widehat\xi_1}$
  for $k=2,3$, so
      \[
d\widehat\xi_2\,d\widehat\xi_3=2\widehat\xi_1\,d\psi_2\,d\psi_3,
\qquad
\Omega_\Psi=\Bigl\{\psi_2^2+\psi_3^2\leq {\textstyle\frac{1}{2}}
\widehat\xi_1\Bigr\},
\qquad
\frac{ \widehat r-\widehat\xi_1}{\psi_2^2+\psi_3^2}\in[c_2,1],
 \]
 where $c_2:=2(1+\sqrt{2}
 )^{-1}$, and we used
 $\widehat r-\widehat\xi_1=(\widehat\xi_2^2+\widehat\xi_3^2)/(\widehat\xi_1+\widehat r)$ to get the final relation above.

    Now a calculation using \eqref{g_xi1} yields
the first desired bound \eqref{g_xi_L1} as follows:
      \begin{align*}
        \norm{\pt_{\xi_1}g}{1,\Omega_*}\!
         &\geq \C
               \iiint_{\wom_{2}}
               \E^{\alpha(\widehat\xi_1-\widehat r)}\,
                  \widehat \xi_1^{-2}
                  \Bigl[\alpha\bigl(\widehat r-\widehat\xi_1\bigr)
                        -\frac{\widehat\xi_1}{\widehat r}\,\Bigr]^+
                  d\ve{\widehat\xi}\\
         &\geq \C \int_1^{\frac14{\eps}^{-1}}\!\!\!\!\!\!\widehat\xi_1^{-1}\,
         d\widehat\xi_1\,
       \iint_{\Omega_\Psi}
                              \E^{-\alpha(\psi_2^2+\psi_3^2)}
                              \Bigl[\alpha c_2(\psi_2^2+\psi_3^2)-1\Bigr]^+
                                     \,d\psi_2\,d\psi_3\\
          &\geq \C|\ln\eps|.
      \end{align*}
A similar calculation using \eqref{g_xi2} with $k=2,3$
 yields \eqref{g_eta_L1}; indeed,
      \begin{align*}
        \norm{\pt_{\xi_k}g}{1,\Omega_*}\!
         &\geq \C\iiint_{\wom_2} \E^{\alpha(\widehat\xi_1-\widehat r)}
\bigl(\alpha\widehat r+1\bigr)\,\widehat \xi_1^{-3}|\widehat\xi_k|\,
                             d\ve{\widehat\xi}\\
         &\geq \C
               \int_1^{\frac14{\eps}^{-1}}\!\!\!\!\!\!\!d\widehat\xi_1
               \,\bigl(\alpha\widehat \xi_1+1\bigr)\,\widehat\xi_1^{-3/2}
                \iint_{\Omega_\Psi} \!\!\!\E^{-\alpha(\psi_2^2+\psi_3^2)}
                             |\psi_k|\,
                             d\psi_2\,d\psi_3 
          \geq \C\eps^{-1/2}.
      \end{align*}

      (ii) To show \eqref{g_eta_ball} for $\rho\le 2\eps$,
      we note that
      $\norm{g}{1,1; B(\ve{x},\rho)}\ge \norm{\pt_{\xi_2}g}{1; B(\ve{x},\rho)}$
      so set $\widehat\rho:=\rho/\eps$ and
      consider  the sub-domain
      $\wom_3:=B(\ve{0},\widehat\rho)\cap
\Bigl\{\widehat\xi_2^2\ge\widehat\xi_1^2+\widehat\xi_3^2\Bigr\}$.
Note that in this sub-domain,
$\E^{\alpha(\widehat\xi_1-\widehat r)}\ge c$ and
$\widehat\xi_2\ge\widehat r/\sqrt{2}$ so (\ref{g_xi2}) yields
      \begin{align*}
        \norm{\pt_{\xi_2}g}{1,\Omega_*\cap B(\ve{x},\rho)}\!
         &\geq \C\iiint_{\wom_3}
\widehat r^{-2}\,
                             d\ve{\widehat\xi}\geq \C
               \int_0^{\widehat\rho}\!d\widehat r
         \geq \C\widehat\rho=\C\rho/\eps,
      \end{align*}
which immediately implies \eqref{g_eta_ball} for $\rho\le 2\eps$.

Next, for $\rho\in [2\eps,\frac18]$ consider $\pt_{\xi_2}g$ in the sub-domain
$\wom_2\cap B(\ve{0},\widehat\rho)$. Imitating the calculation in part (i),
one gets
      \begin{align*}
        \norm{\pt_{\xi_2}g}{1,\Omega_*\cap B(\ve{x},\rho)}\!
         &\geq \C\iiint_{\wom_2\cap B(\ve{0},\widehat\rho)}\! \E^{\alpha(\widehat\xi_1-\widehat r)}
\bigl(\alpha\widehat r+1\bigr)\,\widehat \xi_1^{-3}|\widehat\xi_2|\,
                             d\ve{\widehat\xi}\\
         &\geq \C
               \int_1^{\widehat\rho}\!\!\!d\widehat\xi_1
               \,\bigl(\alpha\widehat \xi_1+1\bigr)\,\widehat\xi_1^{-3/2}
                \iint_{\Omega_\Psi} \!\!\!\E^{-\alpha(\psi_2^2+\psi_3^2)}
                             |\psi_2|\,
                             d\psi_2\,d\psi_3 \\
          &\geq \C\widehat\rho^{-1/2},
      \end{align*}
      which completes the proof of \eqref{g_eta_ball} for $\rho\le\frac18$.

(iii)
To obtain (\ref{g_xi2_L1}), we use the sub-domain
$\wom_1\backslash B(\ve{0},\widehat\rho)$, where $\widehat\rho:=\rho/\eps$.
Note that for $\ve{\widehat\xi}\in \wom_1$ one has
$\E^{\alpha(\widehat\xi_1-\widehat r)}\ge c$ and
$\frac{3\widehat\xi_1^2-\widehat r^2}{\widehat r^2}\ge1$.
So using \eqref{g2_xi1_xi1}, one gets
      \begin{align*}
        \norm{\pt_{\xi_1}^2g}{1,\Omega_*\setminus B(\ve{x},\rho)}
         &\geq \C\eps^{-1}
            \iiint_{\wom_1\backslash B(\ve{0},\widehat\rho)}
               \widehat r^{-3}\,
               \Bigl[0-4\alpha\widehat r+1\Bigr]^+
               d\ve{\widehat\xi}\\
         &\geq \C\eps^{-1}
            \int_{\widehat\rho}^{\min\{1, \frac1{8\alpha}\}}
               \widehat r^{-1}\,
               d\widehat r
          \geq \C\eps^{-1}\ln(2+\rho/\eps).
      \end{align*}
So we have shown (\ref{g_xi2_L1}) for
$\rho\le c_1\eps$ with $c_1:=\frac12\,\min\{1, \frac1{8\alpha}\}$.

(iv)
In a similar manner as in part (iii), using \eqref{g2_xi2_xi2}, one can show that
$\norm{\pt_{\xi_k}^2g}{1,\Omega\setminus B(\ve{x},\rho)}\geq \C\eps^{-1}\ln(2+\rho/\eps)$
for $k=2,3$ and $\rho\le \frac12\eps$.
Note that now we use the sub-domain
$B(\ve{0},1)\cap
\Bigl\{\widehat\xi_k^2\ge\widehat\xi_1^2+\widehat\xi_j^2\Bigr\}$
instead of $\wom_1$, with $j=3$ for $k=2$ and $j=2$ for $k=3$.

Consequently, to obtain (\ref{g_eta2_L1}) for any $\rho\le \frac18$, it remains to
show that
$\norm{\pt_{\xi_k}^2g}{1,\Omega_*\setminus B(\ve{x},\rho)}\geq c\eps^{-1}|\ln\eps|$.
For this, consider \eqref{g2_xi2_xi2} in
$\wom_2\backslash B(\ve{0},\widehat\rho)$.
Combining the observations made in part (i) with
$\alpha\widehat r+1\le (\alpha\sqrt{2}+1)\widehat\xi_1$
and $\frac{3\widehat\xi_k^2-\widehat r^2}{\widehat r^2}\ge-1$,
yields
      \begin{align*}
        \norm{\pt_{\xi_k}^2g}{1,\Omega_*\setminus B(\ve{x},\rho)}
        \hspace{-2.2cm}\\
         &\geq \C\eps^{-1}
            \iiint_{\wom_2\backslash B(\ve{0},\widehat\rho)}
               \E^{\alpha(\widehat\xi_1-\widehat r)}\,
                  \widehat \xi_1^{-3}
                  \Bigl[2\alpha^2{\widehat\xi_1}{\psi_k^2}
                  -(\alpha\sqrt{2}+1)\widehat\xi_1
\,\Bigr]^+
                  d\ve{\widehat\xi}\\
         &\geq \C \int_{\max\{1,\widehat\rho\}}^{\frac14{\eps}^{-1}}\!
         \widehat\xi_1^{-1}\,
         d\widehat\xi_1\,
       \iint_{\Omega_\Psi}
                              \E^{-\alpha(\psi_2^2+\psi_3^2)}
                              \Bigl[{2}\alpha^2{\psi_k^2}
                  -(\alpha\sqrt{2}+1)\Bigr]^+
                                     d\psi_2\,d\psi_3\\
          &\geq \C\eps^{-1}|\ln\eps|.
      \end{align*}
Here we also used $\widehat\rho\le\frac18\eps^{-1}$.
So we have proved the final desired assertion~(\ref{g_eta2_L1}).
\qed
\end{proof}

\section{Approximation of the Green's function
and proof of Theorem~\ref{thm:main_lower} for the domain $\Omega_*=(0,1)\times\R^2$}\label{sec:approx}
   To approximate the Green's function, we use
   the fundamental solution $g$ of Section~\ref{sec:fundamental} and
   the cut-off function $\omega$, defined by
     \begin{equation}\label{omega_def}
         \omega(t) \in C^2(0,1),
         \quad\omega(t)=0\;\;\mbox{for }t\leq \frac{1}{6},
         \quad\omega(t)=1\;\;\mbox{for }t\geq \frac{1}{3},
      \end{equation}
   so 
   $\omega'(x)=0$ for $x=0,1$.
   Then we set
   \[
      \bar {G}(\ve{x};\ve\xi)\!
        =\!\frac{\E^{\alpha\widehat\xi_{1,[ x_1]}}}{4\pi\eps^2}
           \left\{
              \left[ \frac{\E^{-\alpha\widehat r_{[  x_1]}}}{\widehat r_{[  x_1]}}
                    \!-\!\frac{\E^{-\alpha\widehat r_{[ -x_1]}}}{\widehat r_{[ -x_1]}}\right]
             \!-\!\left[ \frac{\E^{-\alpha\widehat r_{[2-x_1]}}}{\widehat r_{[2-x_1]}}
                    \!-\!\frac{\E^{-\alpha\widehat r_{[2+x_1]}}}{\widehat r_{[2+x_1]}}\right]\omega(\xi_1)\!
           \right\},
   \]
   which approximates the Green's function $G$ associated with the domain
   $\Omega_*=(0,1)\times\R^2$; in particular, it
   satisfies the boundary condition
   $\bar {G}\bigr|_{\pt\Omega_*}=\bar {G}\bigr|_{\xi_1=0,1}=0$.
%
%
   Using the notations
%
   \[
      g_{[d]}
        := g(d,x_2,x_3;\ve\xi),\quad
      \lambda^{\pm}:=\E^{2\alpha(1\pm x_1)/{\eps}},\quad
      p:=\E^{-2\alpha x_1/{\eps}},
   \]
%
   we can rewrite the definition of $\bar G$ as
   \begin{align}
      \bar{G}(\ve{x};\ve\xi)
       &= \left[g_{[x_1]}-p\, g_{[-x_1]}\right]
          -\left[\lambda^- g_{[2-x_1]}-p\,\lambda^{\!+} g_{[2+x_1]}\right]\omega(\xi_1).
          \label{bar_G_g}
   \end{align}
   We now present a version of \cite[Lemma 4.2]{FK10_TR1}, which gives
   certain upper bounds for $g$
   that involve a weight function $\lambda$ of type $\lambda^\pm$.
   \begin{lemma}\label{lem:g_lmbd_bounds}
      Let $\ve{x}\in[1+\eps,3]\times\R^2$.
      Then for the function $g$ of \eqref{eq:def_g0}
      and the weight $\lambda:=\E^{2\alpha(x_1-1)/\eps}$
      one has the following bounds
      \begin{subequations}\label{g_bounds_lmbd_}
      \begin{align}
         \norm{(\lambda g)(\ve{x};\cdot)}{1\,;\Omega_*}
            &\leq C\eps,\label{g_L1_mu}\\
         \norm{(\lambda\,\pt_{\xi_k} g)(\ve{x};\cdot)}{1\,;\Omega_*}
            &\leq C, \qquad\;\; k=1,2,3,\label{g_L1_mu2}\\
            \label{g_L1_mu6}
          \norm{(\lambda\,\pt^2_{\xi_k} g)(\ve{x};\cdot)}{1\,;\Omega_*}
               &\leq C\eps^{-1}, \quad k=1,2,3,
      \end{align}
      while for any ball $B(\ve{x'};\rho)$ of radius $\rho\le \frac18$
      centred at $\ve{x'}\in[0,\frac34]\times\R^2$, one has
      \begin{gather}\label{g_L1_mu888}
          \norm{(\lambda\, g)(\ve{x};\cdot)}{1,1\,;\Omega_*\cap B(\ve{x'};\rho)}
               \leq C\min\{\rho/\eps,(\rho/\eps)^{1/2}\}\,\E^{-\frac18\alpha/\eps}.
      \end{gather}
      \end{subequations}
   \end{lemma}
   \begin{proof}
   The bounds (\ref{g_L1_mu}), (\ref{g_L1_mu2}) appear in \cite[Lemma 4.2]{FK10_TR1}.
     The estimate \eqref{g_L1_mu6} is slightly sharper compared to the similar bound in
     \cite[Lemma 4.2]{FK10_TR1}; the latter
     is for the domain $\Omega_*\setminus B(\ve{x},\rho)$ and involves the logarithmic term
     $\ln(2+\eps/\rho)$ as it is valid for $x_1\in[1,3]$.
In the above Lemma~\ref{lem:g_lmbd_bounds} we make a stronger assumption that
$x_1\in[1+\eps,3]$, under which
$\Omega_*\setminus B(\ve{x},\eps)=\Omega_*$ and
the logarithmic term becomes $\ln 3$ so can be dropped.

To obtain the final desired bound (\ref{g_L1_mu888}), note that
$x_1\ge 1+\eps$ implies that $\widehat r >1$ in $\Omega_*$ so
\eqref{eq:def_g0}, \eqref{g_xi1} and \eqref{g_xi2} yield
$\lambda\, (|g|+|\pt_{\xi_k} g|)\le C\eps^{-3}\lambda\, \E^{\alpha
(\widehat\xi_1-\widehat r)}$
for $k=1,2,3$.
Here
$\widehat\xi_1-\widehat r=-(|\widehat\xi_1|+\widehat r)\ge -2|\widehat\xi_1|$
so $\lambda\, \E^{\alpha
(\widehat\xi_1-\widehat r)}\le \E^{2\alpha
[(x_1-1)/\eps-|\widehat\xi_1|]}$.
Note also that
in $B(\ve{x'};\rho)$ we have $\xi_1\le 1-\frac18$ so
$|\widehat\xi_1|=\frac{x_1-\xi_1}\eps\ge \frac{x_1-1}\eps+\frac18\eps^{-1}$.
Consequently,
$\lambda\, (|g|+|\pt_{\xi_k} g|)\le C\eps^{-3}\lambda\, \E^{-\frac14\alpha/\eps}$,
so
$\norm{(\lambda\, g)(\ve{x};\cdot)}{1,1\,;\Omega_*\cap B(\ve{x'};\rho)}
               \leq C(\rho/\eps)^{3}\,\E^{-\frac14\alpha/\eps}$.
Combining this with $(\rho/\eps)^{3}\le C(\rho/\eps)^{p}\,\E^{\frac18\alpha/\eps}$
for $p=1,\frac12$
yields (\ref{g_L1_mu888}).
  \qed
   \end{proof}

   \begin{lemma}\label{lem_barG}
Let $\Omega_*:=(0,1)\times\R^2$,
$\eps\in(0,c_0]$ for some sufficiently small constant $c_0$,
       and $0<\alpha\le C$.
      Then the function $\bar G$ of \eqref{bar_G_g}
      satisfies, for any $\ve{x}\in[1/4,3/4]\times\R^2$,
      all the  bounds~(\ref{g_bounds}) of Lemma~\ref{lem:fundamental}
      with $g$ replaced by $\bar G$.
   \end{lemma}
   \begin{proof}
      Let $D$ be any of the first- or second-order
      differential operators that appear in (\ref{g_bounds}).
      Now for any $\Omega_*'\subset\Omega_*$,
      the representation \eqref{bar_G_g} yields
      \begin{align}
        \norm{D\bar G(\ve{x};\cdot)}{1,\Omega_*'}
         &= \norm{Dg_{[x_1]}-p\, Dg_{[-x_1]}-\lambda^- Dg_{[2-x_1]}+p\,\lambda^{\!+} Dg_{[2+x_1]}}{1,\Omega_*'}
         \notag\\
         &\geq     \norm{Dg_{[x_1]}}{1,\Omega_*'}
                     -p\,\norm{Dg_{[-x_1]}}{1,\Omega_*'}
                     -   \norm{\lambda^-Dg_{[2-x_1]}}{1,\Omega_*'}
                     \notag\\&\hspace*{2.3cm}\,
                     {}-p\,\norm{\lambda^{\!+}Dg_{[2+x_1]}}{1,\Omega_*'}.
                     \label{D_terms}
      \end{align}
      For the first term that involves $g_{[x_1]}$, we use the corresponding lower bound from
      Lemma~\ref{lem:fundamental}
      so it remains to show that this bound will dominate the remaining three terms.
For the second term we note that
$g_{[-x_1]}$ satisfies
the upper bounds of type (\ref{eq_theorem}) with $\Omega$ replaced by $\Omega_*$
\cite{FK10_TR1}. Now, as $x_1\ge\frac14$ implies that $p\le \E^{-\frac12\alpha/\eps}$,
the second term will be dominated by the first term if
$\eps$ is sufficiently small
(i.e. if the constant $c_0$
is sufficiently small).

To estimate the terms that involve $\lambda^\pm g_{[2\pm x_1]}$
in (\ref{D_terms}), we use the bounds~(\ref{g_bounds_lmbd_})
of Lemma~\ref{lem:g_lmbd_bounds}.
In particular, by (\ref{g_xi2_L1}), (\ref{g_L1_mu6}),
\begin{align*}
{\textstyle\frac13}\norm{\pt^2_{\xi_1} g(\ve{x};\cdot)}{1\,;\Omega_*\setminus B(\ve{x};\rho)}
-\norm{\lambda^\pm \pt^2_{\xi_1}g_{[2\pm x_1]}}{1,\Omega_*\setminus B(\ve{x};\rho)}
\hspace{-2.3cm}\\
&\ge\eps^{-1}\bigl[\C\ln(2+\eps/\rho)-C\bigr]\\
&\ge\C\eps^{-1}\ln(2+\eps/\rho)
\end{align*}
for $\rho\le c_1\eps$ if $c_1$ is sufficiently small,
so we get a version of (\ref{g_xi2_L1}) for $\bar G$.
Similarly, by (\ref{g_eta2_L1}), (\ref{g_L1_mu6}), for $k=2,3$ one gets
\begin{align*}
{\textstyle\frac13}\norm{\pt^2_{\xi_k} g(\ve{x};\cdot)}{1\,;\Omega_*\setminus B(\ve{x};\rho)}
-\norm{\lambda^\pm \pt^2_{\xi_k}g_{[2\pm x_1]}}{1,\Omega_*\setminus B(\ve{x};\rho)}
\hspace{-3.3cm}\\
&\ge\eps^{-1}\bigl[\C\ln(2+\eps/\rho)+\C|\ln\eps|-C\bigr]\\
&\ge\C\eps^{-1}(\ln(2+\eps/\rho)+|\ln\eps|)
\end{align*}
provided that $\eps$ is sufficiently small.
This yields a version of (\ref{g_eta2_L1}) for $\bar G$.
Finally,
 \eqref{g_L1_mu888} implies that
\begin{align*}
          \norm{(\lambda^\pm\, g_{[2\pm x_1]})(\ve{x};\cdot)}{1,1\,;\Omega_*\cap B(\ve{x};\rho)}
               &\leq C\min\{\rho/\eps,(\rho/\eps)^{1/2}\}\,\E^{-\frac18\alpha/\eps}\\
               &\leq
               {\textstyle\frac13}\C\min\{\rho/\eps,(\rho/\eps)^{1/2}\}
\end{align*}
for any arbitrarily small $\C$ provided that $\eps$ is sufficiently small.
This observation yields a version of (\ref{g_eta_ball}) for $\bar G$.
      \qed
   \end{proof}

   \textit{Proof of Theorem~\ref{thm:main_lower} for the domain $\Omega_*=(0,1)\times\R^2$.}
As, by Lemma~\ref{lem_barG}, the approximation $\bar G$ satisfies the bounds that
we need to prove for $G$, it suffices to estimate
 the function
   $ v=\bar G- G $,
   which satisfies the differential equation
   \begin{equation}\label{eq:pde_v}
     \LL_{\ve{\xi}}^*\,v(\ve{x};\ve{\xi})
       =-\eps\laplace_{\ve{\xi}}v(\ve{x};\ve{\xi})
        +2\alpha\,\pt_{\xi_1} v(\ve{x};\ve{\xi})=\phi(\ve{x};\ve{\xi})
        \quad\mbox{for}\;\;\ve{\xi}\in\Omega_*,
  \end{equation}
   and the boundary condition $v(\ve{x};\ve{\xi})\bigr|_{\ve{\xi}\in\pt\Omega_*}\!\!=0$.
Here for the right-hand side~$\phi$, it was shown in \cite[Lemma 5.1]{FK10_TR1}
that
   \begin{gather}\label{eq:phi}
     \norm{\phi(\ve{x};\cdot)}{1;\Omega_*}\leq C\E^{-c_3\alpha/\eps}
   \end{gather}
   for some constant $c_3$.
   In view of \eqref{eq:pde_v}, the function $v$ can be represented
   using the Green's function $G$ as
   $$
     v(\ve{x};\ve{\xi}) = \iiint_{\Omega_*} G(\ve{s};\ve{\xi})\,\phi(\ve{x};\ve{s})\,d\ve{s}\,.
   $$
   So applying $\pt_{\xi_k}^p$ to this representation with $p=0,1,2$, $k=1,2,3$,
   for any sub-domain $\Omega_*'\subset\Omega_*$, one gets
   \begin{align*}
     \norm{\pt_{\xi_k}^p v(\ve{x};\cdot)}{1;\Omega_*'}
      &\leq
      \Bigl(\,\,\sup_{\ve{s}\in\Omega_*}
      \norm{\pt_{\xi_k}^p G(\ve{s};\cdot)}{1,\Omega_*'}\,\Bigr)
           \cdot\norm{\phi(\ve{x};\cdot)}{1,\Omega_*}
           \\
      &\leq\, C\E^{-c_3\alpha/\eps}\cdot \sup_{\ve{s}\in\Omega_*}\norm{\pt_{\xi_k}^p G(\ve{s};\cdot)}{1,\Omega_*},
   \end{align*}
   where we also used (\ref{eq:phi}).
   We shall use the above estimate with $\Omega_*':=\Omega_*$ for $p=0,1$
   and $\Omega_*':=\Omega_*\setminus B(\ve{x},\rho)$ for $p=2$.
As the bounds (\ref{eq_theorem}) of Theorem~\ref{thm:main_upper} remain valid for the domain
$\Omega_*$
and also
$\norm{G(\ve{x};\cdot)}{1,\Omega_*}\le C$
\cite{FK10_TR1, FK11_1}, one now concludes that
\begin{subequations}\label{v_bounds}
\begin{align}\notag
\norm{\pt^2_{\xi_k} v(\ve{x};\cdot)}{1;\Omega_*\setminus B(\ve{x},\rho)}&\le
C\E^{-c_3\alpha/\eps}\cdot\eps^{-1}(\ln(2+\eps/\rho)+|\ln\eps|)\\
& \le c'\, \eps^{-1} \ln(2+\eps/\rho)\quad\mbox{for~}k=1,2,3,
\intertext{where $c'$ is arbitrarily small provided that $\eps$ is sufficiently small.
Similarly}
\norm{ v(\ve{x};\cdot)}{1,1;\Omega_*}&\le
C\E^{-c_3\alpha/\eps}\cdot\eps^{-1/2}
 \le c'\, \eps,
 \intertext{which also implies}
\norm{ v(\ve{x};\cdot)}{1,1;\Omega_*\cap B(\ve{x},\rho)}
 &\le c'\, \min\bigl\{\rho/\eps,(\rho/\eps)^{1/2}\bigr\}.
 \end{align}
 \end{subequations}
 As $c'$ in the bounds (\ref{v_bounds}) can be made arbitrarily small,
 combining them with Lemma~\ref{lem_barG} yields
 the desired bounds of Theorem~\ref{thm:main_lower} for the domain $\Omega_*=(0,1)\times\R^2$.
 \qed

\section{Proof of Theorem~\ref{thm:main_lower} for the domain $\Omega=(0,1)^3$}\label{sec_bounded}
The proof of  Theorem~\ref{thm:main_lower} for the domain $\Omega=(0,1)^3$
is very similar to the above proof for the domain $\Omega_*=(0,1)\times\R^2$
presented in Section~\ref{sec:approx}.
The only difference is that instead of
 the approximation $\bar G$
   for the domain $(0,1)\times\R^2$ we now use
the approximation $\bar G_\cube$ defined by
   \begin{align*}
   \bar G_{\mbox{\tiny$\Box$}}
   (\ve{x};\ve\xi)&:= \bar G(\ve{x};\ve\xi)
            \,-\omega_0(\xi_2)\,\bar G(\ve{x};\xi_1,-\xi_2,\xi_3)
            \,\,-\omega_1(\xi_2)\,\bar G(\ve{x};\xi_1,2-\xi_2,\xi_3),
            \\
             \bar G_\cube(\ve{x};\ve\xi)&:=
             \bar G_{\mbox{\tiny$\Box$}}(\ve{x};\ve\xi) \!-\!\omega_0(\xi_3)\,\bar G_{\mbox{\tiny$\Box$}}(\ve{x};\xi_1,\xi_2,-\xi_3)
           \! -\!\omega_1(\xi_3)\,\bar G_{\mbox{\tiny$\Box$}}(\ve{x};\xi_1,\xi_2,2-\xi_3).
 \end{align*}
Here $\omega_0(t):=\omega(1-t)$ and $\omega_1(t):=\omega(t)$
with $\omega$ defined in (\ref{omega_def}),
so that
$\omega_k(k)=1$ and $\omega_k(1-k)=0$ for $k=0,1$.
This approximation was constructed employing the method of images;
an inclusion of  cut-off
functions ensures that it vanishes on $\pt\Omega$.

All the properties of $\bar G$ given in Section~\ref{sec:approx}
remain valid for this new approximation $\bar G_\cube$
with $\Omega_*$ replaced by $\Omega$ provided that
$\ve{x}\in[\frac14,\frac34]^3$.
We leave out the details and only note that the application of the method of images
in the $\xi_2$- and $\xi_3$-directions
is relatively straightforward  as an inspection of (\ref{eq:def_g0})
shows that in these directions, the fundamental solution $g$
is symmetric and exponentially decaying away from the singular point.
 \qed

\section{The Convection-Reaction-Diffusion Case}\label{sec:reac}
   We now slightly generalize  (\ref{eq:Lu}) by including a reaction term
  with a constant coefficient $\beta\ge 0$:
\begin{subequations}\label{crd}
   \begin{align}
     \tilde \LL_{\ve{x}}u(\ve{x})=-\eps\laplace_{\ve{x}}u(\ve{x})
                         -2\alpha\,\pt_{x_1}\! u(\ve{x})
                         +\beta u(\ve{x})
     &=f(\ve{x})&&\mbox{for }\ve{x}\in\Omega\\
     u(\ve{x})&=0&&\mbox{for }\ve{x}\in\partial\Omega.
   \end{align}
   \end{subequations}
Now the fundamental solution $g$ in $\R^3$ satisfies, for each fixed $\ve{x}\in\R^3$
the following version of (\ref{eq:Green_adj_const}) with the adjoint operator
$\tilde \LL^*_{\ve\xi}$:
   \begin{align*}
     \tilde \LL^*_{\ve\xi} g(\ve{x};\ve\xi)
         =-\eps\laplace_{\ve\xi} g(\ve{x};\ve\xi)+2\alpha\,\pt_{\xi_1} g(\ve{x};\ve\xi)+\beta g(\ve{x};\ve\xi)
        &=\delta(\ve{x}-\ve\xi)\quad\mbox{for~} \ve\xi\in\R^3.
   \end{align*}
   Again imitating a calculation of \cite{CK09,KSh87},
   one gets a version of (\ref{eq:def_g0}):
   \[
     g(\ve{x};\ve{\xi})=
      \frac{1}{4\pi\eps^2}\,
      \frac{\E^{\alpha\widehat\xi_{1,[x_1]}-\gamma\widehat r_{[x_1]}}}{\widehat r_{[x_1]}},
      \qquad\mbox{where}
      \quad
\gamma:=\sqrt{\alpha^2+\eps\beta}.
   \]
In view of $\gamma
                =\alpha+\frac{\beta}{2\alpha}\eps+\ord{\eps^2}$,
                an inspection of the proof of Theorem~\ref{thm:main_lower} shows
                that
 the lower bounds (\ref{eq_thm:main_lower}) remain valid for the
   convection-reaction-diffusion problem (\ref{crd}).
%
\section{Outlook for problems in  $\boldsymbol{n}$ dimensions
}\label{sec:outlook}
It was shown in \cite{FK10_1} that
 the upper bounds of Theorem~\ref{thm:main_upper} remain valid for
a two-dimensional variable-coefficient version
 of \eqref{eq:Lu}. Note that
 the fundamental solution
$g_{\R^2}$ that solves (\ref{eq:Green_adj_const}) in $\R^2$,
and also its derivatives involve
   the modified Bessel functions of second kind of order zero $K_0(\cdot)$ and order one
   $K_1(\cdot)$. This fundamental solution is given by
   \vspace{-0.1cm}
   \[
      g_{\R^2}(\ve{x};\ve{\xi}) = \frac{1}{2\pi\eps}\E^{\alpha\widehat\xi_1}K_0(\alpha\widehat r)
   \vspace{-0.1cm}\]
   with the notations $\ve{x},\,\ve{\xi}$, and $\widehat r$ appropriately adapted.
   Using this explicit representation, one can imitate
   the proof of Theorem~\ref{thm:main_lower}
   and get similar lower bounds for the two-dimensional case.
   A certain difficulty lies in having to deal with
the Bessel functions, for which one can simply employ asymptotic
   expansions \cite{AS64,NIST10}:
   \vspace{-0.1cm}
   \begin{align*}
   K_0(z) &= \left(\frac{\pi}{2z}\right)^{1/2}\E^{-z}\,\Bigl(1+\ord{z^{-1}}\Bigr)&&\mbox{for }|z|\gg1,
   \\[-0.1cm]
      K_1(z) &= K_0(z)\Bigl(1+\frac{1}{2z}+\ord{z^{-2}}\Bigr)&&\mbox{for }|z|\gg1,\\[-0.1cm]
K_0(z)&=-\ln z+\ord{1},\qquad
K_1(z)= \frac1z+\ord{1}&&\mbox{for }|z|\ll 1.
   \end{align*}
In this manner one gets the following result.
\begin{theorem}
The Green's function associated with problem \eqref{eq:Lu} in the
unit-square domain $\Omega:=(0,1)^2$, satisfies
a two-dimensional version of Theorem~\ref{thm:main_lower}.
\end{theorem}

   Finally, let us take a look at the problem \eqref{eq:Lu}
   in the $n$-dimensional domain $\Omega:=(0,1)^n$ of
   an arbitrary dimension $n\geq 2$.
   The corresponding fundamental solution $g_{\R^n}$ is given by
   \[
      g_{\R^n}(\ve{x};\ve\xi)
         =\frac{1}{(2\pi)^{n/2}\,\eps^{n-1}}\,\left(\frac{\alpha}{\hat r}\right)^{n/2-1}
            \E^{\alpha\hat\xi_1}K_{n/2-1}(\alpha\hat r)
   \]
   with the modified Bessel functions $K_{n/2-1}(\cdot)$ of second kind of (half-)integer
   order $n/2-1$,
   and the notations $\ve{x},\,\ve{\xi}$, and $\widehat r$ appropriately adapted. Using asymptotic expansions of these Bessel functions \cite{AS64,NIST10},
one can again get a version of Theorem~\ref{thm:main_lower}.

%
   \bibliographystyle{plain}
   \bibliography{lower}

\begin{thebibliography}{10}

\bibitem{AS64}
M.~Abramowitz and I.~A. Stegun.
\newblock {\em {Handbook of Mathematical Functions with Formulas, Graphs and
  Mathematical Tables}}.
\newblock Applied Mathematics Series. National Bureau of Standards,
  Washingtion, D.C., 1964.

\bibitem{CK09}
N.M. Chadha and N.~Kopteva.
\newblock {Maximum norm a posteriori error estimate for a 3d singularly
  perturbed semilinear reaction-diffusion problem}.
\newblock {doi: 10.1007/s10444-010-9163-2 (published online 1 June 2010)},
  2010.

\bibitem{erikss}
K.~Eriksson.
\newblock {An adaptive finite element method with efficient maximum norm error
  control for elliptic problems}.
\newblock {\em Math. Models Methods Appl. Sci.}, 4:313--329, 1994.

\bibitem{FK10_NA}
S.~Franz and N.~Kopteva.
\newblock {A posteriori error estimation for a convection-diffusion problem
  with characteristic layers}.
\newblock in preparation, 2011.

\bibitem{FK10_TR1}
S.~Franz and N.~Kopteva.
\newblock {Full Analysis of Green's function estimates for a
  convection-diffusion problem with characteristic boundary layers in 3d}.
\newblock Technical report, University of Limerick, 2011.

\bibitem{FK11_1}
S.~Franz and N.~Kopteva.
\newblock {Green's Function Estimates for a Convection-Diffusion Problem in
  Three Dimensions}.
\newblock in preparation, 2011.

\bibitem{FK10_1}
S.~Franz and N.~Kopteva.
\newblock {Green's function estimates for a singularly perturbed
  convection-diffusion problem}.
\newblock submitted, 2011.

\bibitem{Leyk}
J.~Guzm\'{a}n, D.~Leykekhman, J.~Rossmann, and A.~H. Schatz.
\newblock {H\"older estimates for Green's functions on convex polyhedral
  domains and their applications to finite element methods}.
\newblock {\em Numer. Math.}, 112:221--243, 2009.

\bibitem{KSh87}
R.~B. Kellogg and S.~Shih.
\newblock {Asymptotic analysis of a singular perturbation problem}.
\newblock {\em SIAM Journal on Mathematical Analysis}, 18(5):1467--1510, 1987.

\bibitem{Kopt08}
N.~Kopteva.
\newblock {Maximum norm a posteriori error estimate for a 2d singularly
  perturbed reaction-diffusion problem}.
\newblock {\em SIAM J. Numer. Anal.}, 46:1602--1618, 2008.

\bibitem{notch}
R.~H. Nochetto.
\newblock {Pointwise a posteriori error estimates for elliptic problems on
  highly graded meshes}.
\newblock {\em Math. Comp.}, 64:1--22, 1995.

\bibitem{NIST10}
F.W.J. Olver, D.W. Lozier, R.~F. Boisvert, and C.W. Clark.
\newblock {\em {NIST Handbook of Mathematical Functions}}.
\newblock Cambridge University Press, Cambridge, 2010.

\end{thebibliography}
\end{document}